\documentclass[a4paper]{scrartcl}

\usepackage{iftex}

\ifPDFTeX 
  \usepackage[utf8]{inputenc}
  \usepackage[T1]{fontenc}
  \usepackage{lmodern}
\else
  \usepackage[no-math]{fontspec}
\fi

\usepackage[mathjax,latexmk]{lwarp}
\HTMLLanguage{en-GB}
\CSSFilename{custom.css}

\usepackage[british]{babel}
\usepackage{etoolbox,xifthen,xparse}

\usepackage[dvipsnames]{xcolor}
\usepackage{aliascnt,enumitem}

\usepackage{amsmath}
\usepackage{amsfonts,amssymb,amsthm,mathtools,booktabs}

\ifPDFTeX
  \usepackage{bbm}
  \usepackage{mathrsfs}
\else
  \usepackage{unicode-math}
  \let\mathbbm\mathbb
\fi
\usepackage[retainorgcmds]{IEEEtrantools}

\usepackage{url,doi}
\usepackage[numbers]{natbib}

\usepackage{todonotes}

\usepackage{hyperref,bookmark}

\hypersetup{hidelinks}

\newcommand{\newsstheorem}[2]{
  \newaliascnt{#1}{dummy}
  \newtheorem{#1}[#1]{#2}
  \aliascntresetthe{#1}
  \expandafter\def\csname #1autorefname\endcsname{#2}
}
\theoremstyle{plain}
\newsstheorem{theorem}{Theorem}

\newsstheorem{proposition}{Proposition}
\newsstheorem{corollary}{Corollary}
\newsstheorem{lemma}{Lemma}

\theoremstyle{definition}
\newsstheorem{definition}{Definition}
\newsstheorem{example}{Example}

\theoremstyle{remark}
\newsstheorem{remark}{Remark}

\setlist[enumerate,1]{label={(\roman*)}}
\setlist[enumerate,2]{label={(\alph*)}}
\setlist[enumerate,3]{label={(\Roman*)}}

\newcommand*\printandhtml[1]{%
  #1%
  \CustomizeMathJax{#1}%
}

\printandhtml{
  \newcommand\dd{\mathrm{d}}
  
  \newcommand\RR{\mathbb{R}}
  
  \newcommand\PP{\mathbb{P}}
  \newcommand\EE{\mathbb{E}}

  \DeclareMathOperator{\sgn}{sgn}
  \newcommand\rhoh{\hat{\rho}}
  \newcommand\HG{{\protect\vphantom{F}}_{2}F_1}
  \DeclareMathOperator{\tRe}{Re}
}
\newcommand\Ind{\mathbbm{1}}
\CustomizeMathJax{
  \newcommand\Ind{\mathbb{1}}
  \newcommand\MoveEqLeft{} % compatibility with mathtools
  \newcommand\mathbbm[1]{\mathbb{#1}}
}

\providecommand{\email}[1]{\href{mailto:#1}{\nolinkurl{#1}}}

\title{Optimal stopping of the stable process with state-dependent killing}
\author{%
  K.\ van Schaik\footnote{University of Manchester, UK.
  \email{kees.vanschaik@manchester.ac.uk}} 
  \and
  A.\ R.\ Watson\footnote{University College London, UK. 
  \email{alexander.watson@ucl.ac.uk}}
  \and
  X.\ Xu\footnote{Xi'an Jiaotong-Liverpool University, PR China.
  \email{Xin.Xu03@xjtlu.edu.cn}}%
}
\date{\today}

\begin{document}
\maketitle

\begin{abstract}
  We describe the solution of an optimal stopping problem for a stable
  Lévy process killed at state-dependent rate.
  The killing rate is chosen in such a way that the killed process
  remains self-similar, and the solution to the optimal stopping
  problem is obtained by characterising a self-similar Markov process
  associated with the stable process.
  The optimal stopping strategy is to stop upon first passage into
  an interval, found explicitly in terms of the parameters of the model.

  \medskip
  {\small
    \noindent
    \emph{2020 Mathematics Subject Classification.}
    60G40, 60G51.
  }
\end{abstract}

\section{Introduction}

Consider a stable Lévy process $X = (X_t)_{t\ge 0}$ 
with index $\alpha\in (0,2)$, to which we
introduce a state-dependent killing, occuring at rate $\omega(X_t)$ at time
$t$, where
\begin{equation*}
  \omega(x)
  =
  \begin{cases}
    k(-x)^{-\alpha}, &x< 0\\
    0, &x\ge 0,
  \end{cases}
\end{equation*}
for some parameter $k > 0$.
Denote the random killing time by $T$, and the killed process by
\[
  X^\dag_t = X_t \mathbbm{1}_{\{t<T\}}.
\]
We consider the gain function
\begin{equation*}
  g(x)
  =
  \begin{cases}
    (x^{r}-K)^{+}, & x > 0\\
    0, & x \le 0,
  \end{cases}
\end{equation*}
for some $r \in \RR \setminus \{0\}$ and $K > 0$, and wish to solve the optimal
stopping problem
\begin{align}
  v(x)
  =
  \sup_{\tau} 
  \mathbb{E}_{x}\left [
    g(X^\dag_{\tau})
  \right ],
  \label{ogp4X}
\end{align}
where the supremum is over all (possibly infinite) stopping times $\tau$ for
the natural enlargement of the filtration of $X^\dag$, and the probability
measure $\PP_x$ with expectation $\EE_x$ indicates that $X_0 = x$.

Optimal stopping problems for Lévy processes have been studied for many
years, but the introduction of state-dependent killing brings new complexity to the
problem. \citet{ZH} have considered a similar question for general Lévy
processes without upward jumps (and with a different gain function $g$ and
killing rate $\omega$). They applied classical techniques in novel ways to
characterise a stopping region which has quite a complex form.  We take a
very different approach: the particular form of $\omega$ means that the process
$X^\dag$ inherits from $X$ the self-similarity property characteristic of
stable processes, and we show that $X^\dag$ has an intimate relationship with a
member of a newly-discovered class of Lévy processes, the double hypergeometric
processes.

This relationship makes it possible to find a surprisingly explicit solution
to the optimal stopping problem: it is optimal to stop when $X^\dag$ enters an
interval, the bounds of which can be given explicitly in terms of the parameters
of the model. This is outlined in the theorem below, which is the main result
of this work and which we initially present only in its broad strokes.
\begin{theorem}\label{t:main}
  There exists $\delta > \max(0,\alpha-1)$, uniquely characterised
  in terms of the parameters of the stable process and the killing
  coefficient $k$, such that the following holds for the process
  $X$ started in $x \ne 0$.
  \begin{enumerate}
    \item 
      When $0< r<\delta$, the solution of the optimal stopping
      problem \eqref{ogp4X} is given by the upwards first passage time
      \[ \tau^{*}\coloneqq \inf \{t\geq 0: X_{t}\geq b^{*}\}, \]
      where $b^*$ can be found explicitly.
    \item When $-(\delta - \alpha + 1)<r<0$,
      the solution of the optimal stopping problem \eqref{ogp4X} is given
      by the first entrance time
      \[ \tau^{*}\coloneqq \inf \{t\geq 0: 0<X_{t}\leq 1/b^*\} \]
      where
      $b^*$ can again be found explicitly.
    \item When
      $r<-(\delta-\alpha+1)$ or $r>\delta$,
      $v(x)=\infty$ for all $x \ne 0$.
  \end{enumerate}
\end{theorem}
In both instances, the quantity $b^*$ can be found in terms of the parameters
of the stable process, the killing parameter $k$ and the parameters
$r,K$ of the gain function,
and there is an explicit expression for $\delta$.
The full
version of this result appears as \autoref{t:os-X-full}, once we have
introduced the necessary notation, and an integral expression for
$v$ is given in \autoref{r:int}.

This work has two purposes: firstly, to describe the solution of an optimal
stopping problem for a stable Lévy process with an omega-clock;
and secondly, to explicitly characterise a new self-similar Markov process,
the killed path-censored stable process.

In the context of stochastic optimal control, the omega-clock was introduced by
\citet{Al}, who studied the dividend problem.  Since then, the model has
enjoyed some popularity as a way to account for bankruptcy: if $X$ represents
the capital of some company, then since the killing rate is only positive when
the process itself is negative, the killing time can be seen as a non-immediate
ruin event. One could also consider $1/X$ to be the capital process (in which
case $\omega$ is decreasing in the capital level), or
indeed look at a time-change of one of these processes by some additive
functional.  The study of the ruin behaviour of risk processes with omega-clock
has been very successful \cite{G,Al2,Ka,Li,Cza}.  Optimal stopping problems
have been considered for this process in the case of Brownian motion
\cite{Kuhn} and of general spectrally negative Lévy processes \cite{ZH}, which
is the situation closest to ours. In \cite{ZH}, \citeauthor{ZH} consider an
American call option with the underlying equal to the exponential of a
spectrally negative Lévy processs, and with an omega-clock function $\omega(x)
= k\Ind_{(-\infty,0)}(x)$ for some $k\ge 0$.  Without the omega-clock, this
problem was solved by \citet{ME}, but with the clock, it becomes substantially
harder, and the authors show that it is optimal to stop upon first passage into
the union of a half-line and a (possibly empty) compact interval.  The main
differences in the present work are that the underlying is a Lévy process (not
its exponential), that the omega-clock function is different in form, and that
the process has both upward and downward jumps.

There is a long tradition of studying stable processes using a transformation
into a positive self-similar Markov process (pssMp); as examples, we mention
the stable process killed on going below zero and its conditionings \cite{CC},
censored \cite{BBC-cens} and path-censored stable processes \cite{KPW} and the
radial part of an isotropic stable process \cite{CPP-explicit}.  Our first step
in dealing with this process with omega-clock is to cut out the excursions
below zero, and this gives rise to a new pssMp, the killed path-censored stable
process, which we describe in explicit detail by relating it to a member of the
double hypergeometric class of Lévy processes.  This allows us to obtain the
solution to the optimal stopping problem.  The relation with a pssMp also
explains the form of the omega-clock function $\omega$, which ensures that the
self-similarity property of the stable process is maintained after killing.

The remainder of the paper is structured as follows. In \autoref{s:kpcsp}, we
construct a pssMp by transforming the path of the killed stable process, and we
apply the Lamperti transform to obtain a new Lévy process from this.  In
\autoref{s:os-lamperti}, we show that the L\'evy process defined in the
preceding section belongs to the double hypergeometric class and solve an
auxiliary optimal stopping problem based on it.  Finally, \autoref{s:os-sol}
provides the proof of the main results of the paper.  A short concluding
Section~\ref{s:variants} offers comments on extensions of this work.

\section{The killed path-censored stable process and its Lamperti transform}
\label{s:kpcsp}

Since the gain function $g$ in the optimal stopping problem \eqref{ogp4X}
is zero on $(-\infty,0)$, it is natural to consider removing the path
sections of $X$ where it is negative. In this section, we will show that
this gives rise to a positive self-similar Markov process, and identify
its distribution using the Lamperti transform.

We begin with a formal presentation of Lévy processes, stable processes,
the omega-clock and the theory of self-similarity. This leads to a
description of a new positive self-similar Markov process $Y$,
the killed path-censored stable process, which we then characterise
and which will be vital to our solution, in \autoref{s:os-lamperti}, of an auxiliary
optimal stopping problem.
The process $Y$ was described briefly in the case of a symmetric stable process
in \cite{DHG}, citing
communication with the second author of this article, and here
we offer a full characterisation in the general case.

A process $\xi=(\xi_t)_{t\geq 0}$ is called a (killed) L\'evy
process if it has state space $\mathbb{R}\cup \{\partial \}$,
càdlàg paths and stationary, independent increments.
Such a process is
characterised by the \emph{L\'evy-Khintchine formula}, which states that for
all $\theta \in \mathbb{R}$, the characteristic exponent given by
$e^{-t\Psi(\theta)}=\mathbb{E}\left [e^{i\theta \xi_{t}}\right ]$ satisfies
\begin{align}
  \Psi(\theta)
  =
  ia\theta
  +\frac{1}{2}\sigma^{2}\theta^{2}
  +\int_{\mathbb{R}}\left (1-e^{i\theta x}+i\theta x\mathbbm{1}_{\{|x|<1\}}\right )\Pi (\mathrm{d}x)
  +q, \label{psixi}
\end{align}
where $q\geq 0$ is the killing rate, $a\in \mathbb{R}$ is the linear
coefficient, $\sigma\geq 0$ is the Gaussian coefficient and $\Pi$ is the L\'evy
measure, concentrated on $\mathbb{R}\backslash \{0\}$ and such that
$\int_{\mathbb{R}}(1\wedge x^{2})\Pi(\mathrm{d}x)<\infty$.
When $q>0$,
the process $\xi$ is sent to the cemetery state $\partial$ at an exponential random time
with rate $q$, which is otherwise independent of the path of $\xi$,
and remains at $\partial$ forever.
We write $\PP_x$ for the law of the process
started from $x$, and we will retain this notation for other
stochastic processes wherever this is unambiguous.

A Lévy process $X=(X_{t})_{t\geq 0}$ 
is called a \emph{stable process} if it enjoys the \emph{scaling property},
namely, that when started from $X_0=0$, the process
$(cX_{tc^{-\alpha}})_{t\geq 0}$ has the same law as $X$ for
any $c>0$. The parameter $\alpha \in (0,2]$ is called the \emph{index} of $X$.
Stable processes can be described in terms of their Lévy-Khintchine formula
as follows
\cite[\S 1.2.6 and \S 6.5.3]{Kyprianou}.
The L\'evy measure $\Pi$ is absolutely continuous with density given
by
\[
  c_{+}x^{-(\alpha+1)}\mathbbm{1}_{\{x>0\}}+c_{-}|x|^{-(\alpha+1)}\mathbbm{1}_{\{x<0\}},
  \quad x\in \mathbb{R},
\]
where 
\[
  c_{+}
  =
  \frac{\Gamma(\alpha+1)}{\Gamma(\alpha \rho)\Gamma(1-\alpha \rho)}
  \text{ and }
  c_{-}
  =
  \frac{\Gamma(\alpha+1)}{\Gamma(\alpha \hat{\rho})\Gamma(1-\alpha \hat{\rho})}.
\]
The parameter $\rho$ is called the \emph{positivity parameter} of $X$
and satisfies $\rho=\mathbb{P}_0(X_{t}\geq 0)$ for all $t>0$.
For convenience, we write $\hat{\rho} = 1-\rho$.
We restrict ourselves to the following set of admissible parameters:
\begin{equation*}
  \mathscr{A}_{st}=\{(\alpha, \rho): \alpha\in (0,1), \rho\in (0,1)\}
  \cup \{(\alpha, \rho):\alpha \in(1,2), \rho\in (1-\frac{1}{\alpha}, \frac{1}{\alpha})\}\cup \{(\alpha, \rho)=(1,\frac{1}{2})\},
\end{equation*}
which encompasses (up to a multiplicative factor in their spatial scale)
all stable processes
with the exception of Brownian motion, processes jumping only in
one direction and symmetric Cauchy processes with non-zero drift.

If $E$ is either $\RR$ or $[0,\infty)$, then
an $E$-valued standard Markov process $Y$
(in the sense of \cite{BG-mp})
is said to be an \emph{E-self-similar Markov process}
if, for some fixed $\alpha>0$ and all $y\in E$ and $c>0$,
\begin{align}
  \text{the law of }
  \left (cY_{c^{-\alpha}t}\right )_{t\geq 0}
  \text{ under }
  \mathbb{P}_{y}
  \text{ is }
  \mathbb{P}_{cy}.
  \label{e:scaling}
\end{align}
When $E = [0,\infty)$, such a process is called a
\emph{positive self-similar Markov process (pssMp)}, and
when $E = \RR$, it is called a \emph{real self-similar Markov process (rssMp)}.
The stable process $X$ is an rssMp.

We are now ready to describe the omega-clock killing of the stable process,
and our procedure for removing its path sections lying in $(-\infty,0)$.

Let $X$ be a stable process.
Due to self-similarity, we can regard the measures $\PP_x$
as being defined by 
$\EE_x[F(X_t, t\ge 0)] = \EE_{\sgn x}[F(\lvert x\rvert X_{t\lvert x\rvert^{-\alpha}}, t\ge 0)]$
for measurable functionals $F$ and $x\in\RR\setminus \{0\}$; this will be convenient at times.
Denote
by $\mathbb{F} = (\mathscr{F}_t)_{t\ge 0}$
the natural enlargement of the filtration $(\sigma(X_s, s\le t))_{t\ge 0}$
in the sense of \citet[Definition~1.3.38]{Bic}.
Note in particular that the enlargement does not depend on the initial
value of $X$.
Define the positive
continuous additive functional
\[
  A_t = \int_0^t \omega(X_s) \, \dd s, \quad t \ge 0,
\]
where $\omega$ is as defined in the introduction.
Let $\mathrm{e}_1$ be an exponential random variable of rate $1$,
independent of $X$, and define the omega-clock killing time
by 
\[
  T = \inf\{ t \ge 0: A_t > \mathrm{e}_1 \}.
\]
The \emph{omega-killed stable process} is given by
\[
  X^\dag_t
  =
  X_t \Ind_{\{t < T\}}, \quad t\ge 0.
\]
We note that here, $0$ functions as a cemetery state for $X^\dag$,
which is a common convention with ssMps.

\begin{lemma}
  \label{l:standard}
  The process $X^\dag$ is a standard Markov process in the filtration
  $\mathbb{F}^\dag$, the natural enlargement of $(\sigma(X^\dag_s, s\le t))_{t\ge 0}$.
\end{lemma}
\begin{proof}
  We give a sketch of the proof, since the argument follows familiar lines.
  Recall that $T<T_0$ almost surely, which implies that,
  for a random time $\tau$, we have $\{ T \le \tau \} = \{X^\dag_\tau = 0 \}$.
  The strong Markov property of $X^\dag$ at the stopping time $\tau$
  can be proved by partitioning on the event $\{X^\dag_\tau = 0\}$ and using
  the corresponding property of $X$ and the definition of $T$.
  Next, the bounded convergence theorem implies that, for any measurable, bounded
  $f$, $\lim_{t\to 0} \EE_x[f(X^\dag_t)] = f(x)$. This and the strong
  Markov property are the only facts required to prove the quasi-left-continuity
  of $X^\dag$, by reproducing the proof of \cite[Lemma~3.2]{Kyprianou}.
  This completes the proof.
\end{proof}

Let $C=\left (C_{t}\right )_{t\geq 0}$ be given by
\[
  C_{t}=\int_{0}^{t}\mathbbm{1}_{\{X^\dag_{s}\ge 0\}}\mathrm{d}s,
  \quad t\ge 0,
\]
and call its right-continuous inverse $C^{-1}$.
$C$ counts the time that $X^\dag$ spends above $0$.
The \emph{killed path-censored stable process} $Y$ is the stochastic process
\[
  Y_{t}
  =
  X^\dag_{C_{t}^{-1}},
  \quad t\ge 0.
\]
The effect of the Markov time-change $C^{-1}$ is to erase
the negative components of $X^\dag$ and glue the non-negative parts
together at the endpoints, up until the time that $X^\dag$ is killed
during one of these negative components.

The role of $0$ deserves special attention.
Write $T_0 = \inf\{t \ge 0: X_t = 0 \}$ for its hitting time.
When $\alpha \le 1$, the process $X$ cannot hit zero,
whereas when $\alpha > 1$, $T_0$ is finite almost surely.
Even in the latter case, $A_{T_0} = \infty$ a.s., and this can be seen
by noting that $k^{-1} A_{T_0}$ is the lifetime of the Lamperti
representation of the path-censored stable process formed from
$-X$, which by \cite[Lemma~13.3]{Kyp} is infinite.
Therefore, $X^\dag$ and $Y$ are always killed before reaching zero.

Moreover, regardless of the value of $\alpha$, when the process
$X$ is started from zero, it is killed immediately.
We can see this is follows. Fix $t$; due to scaling, the distribution
of $k^{-1} A_{t}$ under $\PP_x$ converges, as $x\to 0$, to
that of $k^{-1} A_{T_0}$ under $\PP_1$ (recalling that $T_0<\infty$ iff $\alpha>1$,
and the discussion above).
Applying \cite[Lemma 13.3]{Kyp} again, we see that $A_t=\infty$ a.s.\ for
any $t>0$ when $X_0=0$.
In conclusion, we can
regard $0$ as an absorbing state for $X^\dag$ and $Y$,
which is consistent
with the convention for ssMps mentioned above.

\begin{proposition} 
  \label{p:ssmp}
  The process $Y$ is a positive self-similar Markov process 
  with respect to the filtration 
  $\mathbb{F}^\dag\circ C^{-1} = (\mathscr{F}^\dag_{C^{-1}_t})_{t\ge 0}$.
\end{proposition}

\begin{proof}
  To prove the fact that $Y$ is a standard Markov process, we use \cite{BG-mp},
  specifically the remark following Chapter~V, Proposition~4.11.
  This treats $Y$ as a time-change of $X^\dag$, and one ingredient we require
  is that $X^\dag$ has a reference measure. This is a measure $\mu$ such
  that a set $B\subset\RR$ is null for $\mu$ if and only if it has zero
  potential (meaning in \cite[\S II.3]{BG-mp} that, for some $q\ge 0$ and every $x\in\RR$,
  $U^\dag_q(x,B) \coloneqq \int_0^\infty e^{-qt} \PP_x(X^\dag_t\in B) \, \dd t = 0$).
  We will prove that $\mu(\dd x) = \delta_0(\dd x) + \dd x$, the sum of a Dirac
  mass at zero and Lebesgue measure, is a reference measure for $X^\dag$.
  Fix $q>0$. A short calculation shows that
  \begin{equation}
    U^\dag_q(x,B)
    = \EE_x\biggl[ \int_0^\infty e^{-qt-A_t} \Ind_{\{X_t\in B\}}\, \dd t + (1-e^{-A_t})\delta_0(B) \biggr].
    \label{e:Udag}
  \end{equation}
  From this we see that there exists $a>0$ such that for all $x$ and $B$,
  $U^\dag_q(x,B) \le a(U_q(x,B) + \delta_0(B))$, where $U_q$ is the $q$-potential of
  $X$. Since for every $t>0$, $X_t$ is absolutely continuous with respect to Lebesgue
  measure and has support $\RR$, it follows that $U_q$ is equivalent to Lebesgue measure,
  and thence that $U^\dag_q$ is absolutely continuous with respect to some scalar
  multiple of $\mu$.

  On the other hand, take a measurable set $B\subset \RR$ such that $\mu(B)>0$,
  and fix $x\ne 0$.
  If $\delta_0(B) > 0$, then by \eqref{e:Udag}, $U^\dag_q(x,B)>0$ as well.
  If not, then $B$ has positive Lebesgue measure, and then $U_q(x,B) > 0$ by the
  discussion above. But this implies that $X$ has positive probability of
  reaching $B$ before an exponential time of rate $q$; and $X^\dag$ has positive
  probability of not being killed before reaching $B$. Therefore, $U^\dag_q(x,B) > 0$.
  This completes the proof that $\mu$ is a reference measure for $X^\dag$.

  We now return to the result of \cite{BG-mp}. $Y$ is the time-change of $X^\dag$
  by the continuous additive functional $C$. $X^\dag$ is a standard process
  with a reference measure and $C$ has support $[0,\infty)$, which is closed
  in $\RR$. This is enough to conclude that $Y$ is a standard process.

  We prove the scaling property in two steps: first, we show that
  $X^\dag$ is self-similar, and then that $Y$ inherits this property.

  \begin{enumerate}[label={\arabic*.}]
    \item $X^\dag$ is self-similar; that is, the scaling property
      \eqref{e:scaling} applies to it.
      Fix $c>0$.
      Let $\widetilde{X}_t = cX_{tc^{-\alpha}}$
      and define the rescaled process $\widetilde{A}$ as
      \begin{equation*}
        \begin{split}
          \widetilde{A}_{t}
          = \int_0^t \omega(\widetilde{X}_s)\,\dd s
          &=\int_{0}^{t}
          k c^{-\alpha}\left (-(X_{c^{-\alpha}s})\right )^{-\alpha}\mathbbm{1}_{\{X_s<0\}}
          \mathrm{d}s\\
          &=\int_{0}^{c^{-\alpha}t}
          k(-X_{u})^{-\alpha}\mathbbm{1}_{\{X_u<0\}}
          \mathrm{d}u
          =A_{c^{-\alpha}t}
        \end{split}
      \end{equation*}
      Let $\widetilde{T}=\inf \{t: \widetilde{A}_{t}>\mathrm{e}_1\}
      = c^\alpha T$.
      The scaling property of $X^\dag$ now follows, using in the third line
      the scaling of $X$:
      \begin{equation*}
        \begin{split}
          \text{under }\mathbb{P}_{x},\text{ }
          (cX^\dag_{c^{-\alpha}t})_{t\geq 0}
          =\left (cX_{c^{-\alpha}t}\mathbbm{1}_{\{c^{-\alpha}t<T\}}\right )_{t\geq 0}
          &= \left( \widetilde{X}_t \Ind_{\{t<\widetilde{T}\}}\right)_{t\ge 0} \\
          &\overset{d}{=}
          \left (X_{t}\mathbbm{1}_{\{t< T\}}\right )_{t\geq 0}
          = X^\dag
          \text{ under }
          \mathbb{P}_{cx}.
        \end{split}
      \end{equation*}

    \item $Y$ is self-similar.

      Let $\tilde{C}$ be the functional $C$ applied to the process
      $(c X^\dag_{tc^{-\alpha}})_{t\ge 0}$; a calculation similar to the one
      above yields that
      $\tilde{C}^{-1}_t = c^\alpha C^{-1}_{c^{-\alpha}t}$.
      We deduce the scaling property of $Y$:
      \begin{equation*}
        \begin{split}
          \text{under }
          \mathbb{P}_{x},\  
          (cY_{c^{-\alpha}t})_{t\geq 0}
          = \left (cX^{\dag}_{C^{-1}_{c^{-\alpha}t}}\right )_{t\geq 0}
          &=\left ( cX^\dag_{c^{-\alpha}\tilde{C}^{-1}_t}\right )_{t\geq 0}
          \\
          &\overset{d}{=}\left (X^\dag_{C^{-1}_t}\right )_{t\geq 0}
          = Y
          \text{ under }
          \mathbb{P}_{cx},
        \end{split}
      \end{equation*}
      where we used step 1 in the third equality.
  \end{enumerate}
  Since $Y$ evidently has state space $[0,\infty)$, this completes the proof.
\end{proof}

The work of Lamperti \cite{Lam} provides a bijection between the class 
of L\'evy processes killed at an independent and exponentially distributed time 
and the class of positive self-similar Markov processes, 
which can be expressed through a straightforward space-time transformation;
\cite[\S 13]{Kyprianou} offers a textbook treatment.
Let $T(s) = \left( \int_0^\cdot Y_u^{-\alpha}\, \dd u\right)^{-1}(s)$, and
\begin{align}
  \xi_{s}=\log Y_{T(s)}, \quad s\geq 0.\label{lamperti}
\end{align}
Then, $\xi$ is a Lévy process killed at positive rate.

Our next aim is to obtain the characteristic function of $\xi$, using
the structure of $Y$ in terms of gluing path sections of $X^\dag$.

Define a stopping time 
\[\tau_{0}^{-}=\inf\{t\geq 0: X_{t}<0\},\]
at which the stable process $X$ passes below zero for the first time.
We will denote the first `gluing time' of $Y$ by 
$\sigma_0 = \inf\{ t\ge 0 : C^{-1}_t > C^{-1}_{t-}\}$; in fact,
$\sigma_0 = \tau_0^-$, but the latter notation would be misleading for $Y$,
since it actually stays positive at this time.

\begin{lemma} 
  \label{l:Ysigma}
  For any $x>0$, the joint law of $(Y_{\sigma_{0}}, Y_{\sigma_{0}-})$ under $\mathbb{P}_{x}$ is equal to that of $(xY_{\sigma_{0}}, xY_{\sigma_{0}-})$ under $\mathbb{P}_{1}$.
\end{lemma}

\begin{proof}
  The proof is very similar to \cite[Lemma 3.5]{Alex}, but as the situation is
  slightly different, we include it for completeness.
  Fix $c>0$ and define the two rescaled processes $(\widetilde{X}_{t})_{t\geq 0}$ by $\widetilde{X}_{t}=cX_{c^{-\alpha}t}$ and $(\widetilde{Y}_{t})_{t\geq 0}$ by $\widetilde{Y}_{t}=cY_{c^{-\alpha}t}$. Let $\tilde{\tau}_{0}^{-}=\inf \{t\geq 0: \tilde{X}_{t}<0\}$,
  and analogously define $\tilde{\sigma}_0$, the first gluing time of $\tilde{Y}$, which is
  equal to $\tilde{\tau}_0^-$.
  Then,
  \begin{equation*}
    c^{\alpha}\tau_{0}^{-}=\inf\{ c^{\alpha}t: t\geq 0, X_{t}<0\}=\inf\{t\geq 0: cX_{c^{-\alpha}t}<0\}=\tilde{\tau}_{0}^{-}.
  \end{equation*}
  Therefore, for every $c,x>0$ and $y,z\in\RR$,
  $\PP_x(Y_{\sigma_0-} \in \mathrm{d}y, Y_{\sigma_0} \in \mathrm{d}z)
  = \PP_{cx}( c^{-1}Y_{\sigma_0-}\in \mathrm{d}y, c^{-1} Y_{\sigma_0} \in \mathrm{d}z)$.
  The lemma follows by setting $c=1/x$.
\end{proof}

Denote by $p$ the \emph{killing probability} of $Y$ at the first gluing event, namely
\begin{equation*}
  p
  =
  \mathbb{P}_x\left (Y_{\sigma_{0}}=0\right ),
\end{equation*}
which we assert is independent of $x$.
The following lemma gives the explicit expression for $p$.

\begin{lemma}[Killing probability]
  \label{l:kp}
  The killing probability $p$ is independent of $x$ and is given by
  \[
    p = \frac{k}{c_{+}/\alpha +k}.
  \]
  As a consequence, $T<\infty$.
\end{lemma}

\begin{proof}
  Recall that $T$ is the time at which $X^\dag$ is killed,
  and so $C_T$ is the killing time of $Y$.
  Let $R = \inf\{ t > \tau_0^- : X_t \ge 0 \}$, the first return time
  of $X$ above zero.
  In terms of these quantities, $p = \PP(T \le R)$.

  Consider the dual process $\hat{X}$ with distribution $-X$,
  which is still a stable process (with different parameters).
  Let $\hat{X}^*$ denote the process $\hat{X}$ sent to zero at the
  first time it passes below zero. It is well-known 
  \cite[Theorem~2]{CC} that
  the Lamperti transform of the pssMp $\hat{X}^*$ is killed at
  exponential time of rate $c_+/\alpha$, regardless
  of the value of $\hat{X}^*_0$, and said killing time
  corresponds (through the Lamperti time-change)
  to the first time that $\hat{X}$ passes below zero.

  Let $\hat{\tau}_0^- = \inf\{ t\ge 0: \hat{X}_t < 0\}$.
  The discussion above amounts to the statement that, whatever the initial value
  of $\hat{X}$, the distribution of 
  $\int_0^{\hat{\tau}_0^-} (\hat{X}_u)^{-\alpha} \, \mathrm{d} u$
  is exponential with parameter $c_+/\alpha$.
  This leads us to
  the following calculation, in which $e_1$ is a random variable with exponential
  distribution of rate $1$, independent of $\hat{X}$.
  \begin{align*}
    \mathbb{P}_x(R < T)
    &=
    \mathbb{E}_x \left[
      \mathbb{P}_{-X_{\tau_0^-}}
      \biggl(
        \int_0^{\hat{\tau}_0^-} k(\hat{X}_u)^{-\alpha} \, \mathrm{d} u < e_1 
      \biggr)
    \right]
    \\
    &= 
    \mathbb{E}_x \left[
      \mathbb{P}_{-X_{\tau_0^-}}
      \biggl(
        \int_0^{\hat{\tau}_0^-} (\hat{X}_u)^{-\alpha} \, \mathrm{d} u < e_1/k
      \biggr)
    \right]
    \\
    &=\frac{\frac{c_{+}}{{\alpha}}}{\frac{c_{+}}{\alpha}+k}.
  \end{align*}
  The killing probability is 
  $p = 1 - \mathbb{P}(R<T)$, and this completes the proof.

  Since $p$ is indeed independent of $X$, and since $X^\dag$ has fixed, positive
  probability of being killed in every excursion it makes below the level $0$,
  it must eventually be killed; that is, $T<\infty$.
\end{proof}

Write $X^*$ for the stable process killed upon exiting $[0,\infty)$, and
$\xi^*$ for the Lévy process appearing in its Lamperti representation.
The law of the latter was
characterised in \cite[Theorem 2]{CC}, and its characteristic exponent
computed in \cite[Theorem 1]{KP-hg}.
Making use of the preparatory work above, we are now able to describe the
path structure of $\xi$ and compute its characteristic exponent.

\begin{proposition}[Structure of $\xi$]
  \label{p:structure}
  The L\'evy process $\xi$ is the sum of two independent L\'evy processes $\xi^{1}$ and $\xi^{2}$, which are characterised as follows:
  \begin{enumerate}
    \item
      The L\'evy process $\xi^{1}$ has characteristic exponent
      \[\Psi^{1}(\theta)=\Psi^{*}(\theta)-\frac{c_{-}}{\alpha},\,\,\,\,\theta \in \mathbb{R},\]
      where $\Psi^{*}$ is the characteristic exponent of the process $\xi^{*}$.
    \item
      The process $\xi^{2}$ has has characteristic exponent
      \[\Psi^{2}(\theta)=(1-p)\Psi^{\text{\textnormal{cpp}}}(\theta)+p\frac{c_{-}}{\alpha},\]
      where $p=\frac{k}{c_{+}/\alpha +k}$ is the killing probability and $\Psi^{\text{\textnormal{cpp}}}$ is the characteristic exponent of $\xi^{c}$, a compound Poisson process with jump rate $c_{-}/\alpha$, which is expressed as
      \[
        \Psi^{\text{\textnormal{cpp}}}(\theta)
        =
        \frac{c_-}{\alpha}
        \left (1-\frac{\Gamma(1-\alpha \rho+i\theta)\Gamma(\alpha \rho-i\theta)\Gamma(1+i\theta)\Gamma(\alpha-i\theta)}{\Gamma(\alpha \rho)\Gamma(1-\alpha \rho)\Gamma(\alpha)}\right ),
      \]
      for $\theta\in \mathbb{R}$.
  \end{enumerate}
\end{proposition}
\begin{proof}
  The proof is almost identical to that of Proposition~3.4 in \cite{KPW},
  which operates by considering the path up to the gluing time $\sigma_0$.
  The only difference is that at time $\sigma_0$, the process $X^\dag$
  may be sent to zero, which corresponds to killing of the process $\xi^2$
  in place of a jump. Accounting for this in the obvious way,
  making use of Lemmas~\ref{l:Ysigma} and~\ref{l:kp} in the appropriate place,
  yields the result on the path structure.
  The expression for $\Psi^{\text{cpp}}$ is obtained from
  \cite[Proposition~4.2]{KPW}.
\end{proof}

This structure will allow us to compute the characteristic exponent of $\xi$.
To give the explicit expression, we first require some more notation.

\begin{lemma}
  \label{l:delta}
  Let
  \[
    \delta
    = \frac{1}{2}
    \Bigl( \alpha - 
      \frac{1}{\pi}
      \arccos\bigl( 
        p\cos \pi(\alpha\rho - \alpha\hat{\rho})
        + (1-p)\cos\pi\alpha
      \bigr)
    \Bigr).
  \]
  Then, $\delta$ uniquely satisfies the conditions
  \begin{align}
    \max(0,\alpha-1) < \delta < \min(\alpha\rho,\alpha\hat{\rho})\label{e:gamma-ineq}
  \end{align}
  and
  \begin{align}
    (1-p)\sin \pi\alpha \rho \sin\pi\alpha \hat{\rho}
    =
    \sin\pi(\alpha\rho - \delta) \sin\pi(\alpha\hat{\rho}-\delta).\label{gammadef}
  \end{align}
\end{lemma}
\begin{proof}
  Using product-to-sum identities, condition \eqref{gammadef} can be rewritten
  as follows:
  \begin{align}
    (1-p)\bigl( \cos \pi(\alpha\rho-\alpha\rhoh) - \cos \pi\alpha \bigr)
    &= \cos \pi(\alpha\rho-\alpha\rhoh)-\cos \pi s \nonumber \\
    \cos(\pi s) 
    &= p \cos \pi(\alpha \rho - \alpha\hat{\rho}) + (1-p) \cos \pi\alpha,
    \label{e:gamma-tr}
  \end{align}
  where $s = 2(\alpha/2-\delta)$,
  and the inequalities \eqref{e:gamma-ineq} are equivalent to
  $\max(\alpha\rho-\alpha\hat{\rho},\alpha\hat{\rho}-\alpha\rho) < s < \min(\alpha,2-\alpha)$.

  We divide our analysis into two cases depending on the value of $\alpha$.
  When $\alpha \in (0,1]$, we have that $-\alpha < \alpha\rho - \alpha\hat{\rho} < \alpha$.
  If $\rho \ge 1/2$, then taking 
  \[
    s = \frac{1}{\pi}\arccos\bigl(p \cos \pi(\alpha \rho - \alpha\hat{\rho}) + (1-p) \cos \pi\alpha\bigr)
  \]
  yields
  $0 < \alpha\rho - \alpha\hat{\rho} < s < \alpha \le  1$. Moreover, this is the unique
  value of $s$ in the interval specified which satisfies \eqref{e:gamma-tr}.
  The analysis is similar when $\rho < 1/2$.

  On the other hand, when $\alpha \in (1,2)$, we have instead
  $\alpha-2<\alpha\rho - \alpha\hat{\rho} < 2-\alpha$.
  If $\rho \ge 1/2$, then taking $s$ as above gives 
  $0 < \alpha\rho - \alpha\hat{\rho} < s < 2-\alpha < 1$;
  again, the uniqueness argument and the case $\rho<1/2$ are similar.
\end{proof}

We briefly note a few extensions and special cases.
When $p=1$, which is not part of our parameter set but
corresponds formally to $k=\infty$, that is, immediate killing when
$X$ goes below zero, we have $\delta = \min(\alpha\rho,\alpha\hat{\rho})$.
When $\rho = 1/2$, the symmetric case, we have
$\delta = \frac{1}{2}\bigl(\alpha 
- \frac{1}{\pi}\arccos\bigl(p+(1-p)\cos \pi\alpha\bigr)\bigr)
= \frac{\alpha}{2} - \frac{1}{\pi}\arcsin\bigl(\sqrt{1-p}\sin(\pi\alpha/2)\bigr)$;
this calculation corresponds to the one cited in
\cite{DHG} for the process denoted there by $Y^\natural$.

The structure of $\xi$ allows us to deduce explicitly the characteristic
exponent of the process.

\begin{corollary}
  The characteristic exponent of $\xi$ is expressed as
  \begin{align}
    \Psi(\theta)
    =
    \frac{\Gamma(\alpha-i\theta)\Gamma(\alpha\rho-i\theta)
    \Gamma(1+i\theta)\Gamma(1-\alpha\rho+i\theta)}
    {\Gamma(\alpha-\delta-i\theta)\Gamma(\delta-i\theta)
    \Gamma(\delta+1-\alpha+i\theta)\Gamma(1-\delta+i\theta)}.\label{WHPsi}
  \end{align}
\end{corollary}

\begin{proof}
  The beginning of the proof resembles that of \cite[Theorem 5.3]{KPW},
  but it then diverges due to the killing.
  Using the structure described in \autoref{p:structure}
  and substituting the expression for $\Psi^*$ found in \cite[Theorem 1]{KP-hg},
  we obtain:
  \begin{align}
    \Psi(\theta)
    &= \Psi^{*}(\theta)+(1-p)\Psi^{\text{cpp}}(\theta)-(1-p)\frac{c_{-}}{\alpha}
    \nonumber \\
    &= \frac{\Gamma(\alpha-i\theta)\Gamma(1+i\theta)}
    {\Gamma(\alpha \hat{\rho}-i\theta)\Gamma(1-\alpha\hat{\rho}+i\theta)}
    \nonumber \\
    &\quad {} + (1-p)\frac{c_{-}}{\alpha}
    - (1-p)\frac{c_{-}}{\alpha}
    \frac{\Gamma(1-\alpha\rho+i\theta)\Gamma(\alpha \rho-i\theta)
    \Gamma(\alpha-i\theta)\Gamma(1+i\theta)}
    {\Gamma(\alpha\rho)\Gamma(1-\alpha \rho)\Gamma(\alpha)} 
    \nonumber \\
    &\quad {} - (1-p)\frac{c_{-}}{\alpha}.
    \nonumber
  \end{align}
  Rearranging this and substituting the expression for $c_-$, yields
  \begin{align}
    \Psi(\theta) &= \Gamma(\alpha-i\theta)\Gamma(1+i\theta) 
    \nonumber \\
    &\quad {} \times 
    \left [
      \frac{1}{\Gamma(\alpha \hat{\rho}-i\theta)\Gamma(1-\alpha \hat{\rho}+i\theta)}
      -(1-p)\frac{\Gamma(\alpha \rho-i\theta)\Gamma(1-\alpha \rho +i\theta)}
      {\Gamma(\alpha \rho)\Gamma(1-\alpha \rho)
      \Gamma(\alpha \hat{\rho})\Gamma(1-\alpha \hat{\rho})}
    \right ]  
    \nonumber \\
    &= \frac{1}{\pi^{2}}
    \Gamma(\alpha-i\theta)\Gamma(1+i\theta)
    \Gamma(\alpha \rho-i\theta)\Gamma(1-\alpha\rho+i\theta) 
    \nonumber \\
    &\quad {} \times 
    \left [
      \sin\pi(\alpha\rho-i\theta)\sin\pi(\alpha \hat{\rho}-i\theta)
      -(1-p)\sin\pi\alpha\rho\sin\pi\alpha \hat{\rho}
    \right ],
    \label{psi4sin}
  \end{align}
  where we apply the reflection formula to the last equality. 
  Applying \eqref{gammadef} gives that
  \begin{align}
    \MoveEqLeft \sin\pi(\alpha\rho-i\theta)\sin\pi(\alpha \hat{\rho}-i\theta)-(1-p)\sin\pi\alpha\rho\sin\pi\alpha \hat{\rho}
    \nonumber \\
    &=
    \frac{1}{2}\bigl[
      \cos \pi(\alpha\rho - \alpha\rhoh) - \cos\pi(\alpha-2i\theta) - \cos\pi(\alpha\rho-\alpha\rhoh)
      + \cos\pi(\alpha-2\delta)
    \bigr]
    \nonumber \\
    &= \sin \pi(\alpha-\delta - i\theta) \sin \pi(\delta - i\theta) 
    \nonumber \\
    &=
    \frac{\pi^2}
    {\Gamma(\alpha-\delta-i\theta)\Gamma(1-\alpha+\delta+i\theta)
    \Gamma(\delta-i\theta)\Gamma(1-\delta+i\theta)},
    \label{sin4ga}
  \end{align}
  using product-to-sum and sum-to-product identities followed by the reflection formula.
  Finally, substituting \eqref{sin4ga} into \eqref{psi4sin} yields the expresssion
  in the statement.
\end{proof}

\begin{remark}
  We close the section by describing the situation where $k=0$.  This is not
  part of our assumptions, but there is a link to the
  existing literature.  In this case, noting $T=\infty$,
  one should replace $X^\dag$ with the
  definition $X^\dag_t = X_t\Ind_{\{t < T_0\}}$ to ensure that the state
  $0$ is absorbing.

  The Lévy process $\xi$ is unkilled in this case, and in \autoref{l:delta},
  the lower bound $\delta = \max(0,\alpha-1)$ is attained.  Here we have $p=0$,
  which represents the fact that $Y$ is simply the path-censored stable process
  defined in \cite{KPW}. Many of the arguments that appear in this section are
  analogous to ones in that work, though the presence of killing when $k>0$
  introduces some interesting novel features.
\end{remark}

\section{An optimal stopping problem for the Lamperti transform \texorpdfstring{$\xi$}{xi}}
\label{s:os-lamperti}

Having identified the pssMp $Y$ via its Lamperti transform $\xi$, we are in
a position to solve a related optimal stopping problem for the latter process,
which we will later translate into a solution of the original problem.

The solution to the optimal stopping problem for $\xi$
will rely on the \emph{Wiener-Hopf factorisation}
of Lévy processes. This can briefly be described as follows
in terms of characteristic and Laplace exponents, with the meaning tha
a function $\phi$ is the Laplace exponent of a subordinator $H$
if $\EE e^{-\lambda H_t} = e^{-t\phi(\lambda)}$.
If $\Psi$ is the characteristic exponent of a Lévy process $\xi$
which is killed at rate $q\ge 0$, then there exists a factorisation
of $\Psi$ of the form
\begin{align}
  \Psi(\theta)
  =
  \kappa (q, -i\theta)\hat{\kappa}(q, i\theta),
  \quad \theta \in \mathbb{R},\label{kappa4phi}
\end{align}
where $\kappa(q,\cdot)$ and $\hat{\kappa}(q,\cdot)$ are Laplace exponents
of two (possibly killed) subordinators, known as the ascending
and descending ladder height processes.
The functions $\kappa$ and $\hat{\kappa}$ are called the
\emph{Wiener-Hopf factors} of $\Psi$ (or of $\xi$),
and this factorisation is unique up to multiplication of each
factor by a positive constant.

These subordinators describe the way that $\xi$ makes new maxima and minima,
which goes some way to explaining their utility in the context of this
problem.
A full treatment of the theory of Wiener-Hopf factorisation can be
found in \cite[\S 6]{Kyprianou}.

Our first goal in this section is to characterise $\xi$ by identifying
it as a double hypergeometric Lévy process. This is a recently defined class of
processes with explicit Wiener-Hopf factorisation.
The process $\xi$ is the second known example of a double hypergeometric
process found `in the wild', the other being the ricocheted stable process
described by Budd \cite{Budd}.

\subsection{Identification of the Lamperti transform \texorpdfstring{$\xi$}{xi}}
Double hypergeometric processes, introduced in \cite{DHG}, 
are a family of L\'evy processes with known Wiener-Hopf factorisation.
The class can be characterised as follows.
Let $\mathscr{O}$ be the set of all 
$(\mathtt{a},\mathtt{b},\mathtt{c},\mathtt{d}) \in [0,\infty)^4$
satisfying one of
\begin{enumerate}
  \item\label{i:O-member-1}
    for some $n\in \mathbb{N} \cup\{0\}$,
    $\mathtt{c}+n \le \mathtt{a}+n \le \mathtt{d} \le \mathtt{b} \le \mathtt{c}+n+1$,
    or
  \item
    for some $n\in \mathbb{N}$,
    $\mathtt{a}+n-1 \le \mathtt{c}+n\le\mathtt{b}\le\mathtt{d}\le\mathtt{a}+n$.
\end{enumerate}
When $(\mathtt{a},\mathtt{b},\mathtt{c},\mathtt{d}) \in \mathscr{O}$,
the function
\[
  \mathrm{B}(\mathtt{a}, \mathtt{b}, \mathtt{c}, \mathtt{d}; \lambda)
  \coloneqq 
  \frac{\Gamma (\lambda+\mathtt{a})\Gamma(\lambda+\mathtt{b})}
  {\Gamma(\lambda+\mathtt{c}) \Gamma(\lambda+\mathtt{d})},
  \quad \tRe \lambda \geq 0,
\]
is the Laplace exponent of a subordinator.

Moreover,
it is shown in \cite[Corollary~2.1]{DHG} that,
when $(\mathtt{a}, \mathtt{b}, \mathtt{c}, \mathtt{d}), (\hat{\mathtt{a}}, \hat{\mathtt{b}}, \hat{\mathtt{c}},\hat{ \mathtt{d}}) \in \mathscr{O}$,
the function
\[
  \Psi(\theta)
  =
  \mathrm{B}(\mathtt{a}, \mathtt{b}, \mathtt{c}, \mathtt{d}; -i\theta) 
  \mathrm{B}(\hat{\mathtt{a}},\hat{\mathtt{b}}, \hat{\mathtt{c}}, \hat{\mathtt{d}}; i\theta),
  \quad \theta\in \mathbb{R},
\]
is the characteristic exponent of a Lévy process in the
\emph{double hypergeometric class}.

\begin{lemma}
  The process $\xi$ is a double hypergeometric L\'evy process with parameters
  \[
    \mathtt{a}=\alpha\rho,
    \mathtt{b}=\alpha, 
    \mathtt{c}=\delta,
    \mathtt{d}=\alpha-\delta,
  \]
  and
  \[
    \hat{\mathtt{a}}=1-\alpha \rho,
    \hat{\mathtt{b}}=1, 
    \hat{\mathtt{c}}=\delta+1-\alpha,
    \hat{\mathtt{d}}=1-\delta.
  \]
  Accordingly, the Wiener-Hopf factors of $\xi$ are
  \begin{align}
    \kappa(q,z)
    =
    \frac{\Gamma(\alpha\rho+z)\Gamma(\alpha+z)}
    {\Gamma(\delta+z)\Gamma(\alpha-\delta+z)},\label{kapp}
  \end{align}
  and
  \begin{align}
    \hat{\kappa}(q,z)
    =
    \frac{\Gamma(1-\alpha \rho+z)\Gamma(1+z)}
    {\Gamma(\delta+1-\alpha+z)\Gamma(1-\delta+z)},\label{kapn}
  \end{align}
  where
  \[
    q = 
    \frac{\Gamma(\alpha\rho)\Gamma(\alpha)\Gamma(1-\alpha\rho)}
    {\Gamma(\delta)\Gamma(\alpha-\delta)\Gamma(\delta+1-\alpha)\Gamma(1-\delta)}
    =
    \frac{c_-}{\alpha}\frac{k}{k+\frac{c_+}{\alpha}}
    .
  \]
\end{lemma}
\begin{proof}
  The interval for $\delta$ given in \eqref{e:gamma-ineq} implies that
  both
  $(\mathtt{a},\mathtt{b}, \mathtt{c}, \mathtt{d})$
  and $(\hat{\mathtt{a}}, \hat{\mathtt{b}}, \hat{\mathtt{c}},\hat{ \mathtt{d}})$
  satisfy
  condition \ref{i:O-member-1} for membership of $\mathscr{O}$,
  with $n=0$.

  It follows from Corollary~2.1 in \cite{DHG}
  that the Lévy process $\xi$ with 
  characteristic exponent
  \begin{align*}
    \Psi(\theta)
    & =
    \mathrm{B}(\mathtt{a}, \mathtt{b}, \mathtt{c}, \mathtt{d}; -i\theta) 
    \mathrm{B}(\hat{\mathtt{a}},\hat{\mathtt{b}}, \hat{\mathtt{c}}, \hat{\mathtt{d}}; i\theta),
    \\
    &=
    \frac{\Gamma(\alpha\rho-i\theta)\Gamma(\alpha-i\theta)
    \Gamma(1-\alpha\rho+i\theta)\Gamma(1+i\theta)}
    {\Gamma(\delta-i\theta)\Gamma(\alpha-\delta-i\theta)
    \Gamma(\delta+1-\alpha+i\theta)\Gamma(1-\delta+i\theta)}
  \end{align*}
  exists as a member of the double hypergeometric class.
  Comparing with \eqref{WHPsi} shows that this identifies our process $\xi$,
  and the result of \cite{DHG} yields the stated Wiener-Hopf factorisation.
\end{proof}

When $k=0$, the process $\xi$ is a hypergeometric Lévy process
in the simple ($\alpha\le 1$) or extended ($\alpha > 1$) class,
as described in \cite[\S 5]{KPW}. The expressions given in the above result
for its Wiener-Hopf factors remain valid in this case, though they can
be simplified further.

\subsection{The optimal stopping problem as a perpetual call option}
In this part, we derive the solutions for the optimal stopping problem
\begin{align}
  w(y)=\sup_{\sigma\in \mathscr{S}_{\mathbb{G}}}\mathbb{E}_{y}\left [\left (e^{r\xi_{\sigma}}-K\right )^{+}\right ],\label{callop}
\end{align}
where $r\in \mathbb{R} \setminus \{0\}$,
and $\mathscr{S}_{\mathbb{G}}$ indicates the set of all stopping times 
with respect to $\mathbb{G} = (\mathscr{G}_t)_{t\ge 0}$, the
natural enlargement of $(\sigma(\xi_s, s\le t))_{t\ge 0}$, defined as in
\autoref{s:kpcsp}.
This is a perpetual American call option in which the underlying is
the process $e^{r\xi}$, and has been addressed in \cite{ME}.

The Wiener-Hopf factors $\kappa$ and $\hat{\kappa}$ of $\xi$ appear
as components in the solution.
We note that
$\kappa(q,z)$ is a well-defined holomorphic function
for $\tRe z > -\alpha\rho$, and the same applies to $\hat\kappa(q,z)$
for $\tRe z > \alpha\rho-1$.

\begin{theorem}\label{t:main-lamperti}
  Consider the optimal stopping problem \eqref{callop}.
  \begin{enumerate}
    \item\label{i:xi:a} When $0< r<\delta$,  the solution is given by
      \begin{align}
        w(y)
        =
        \frac{1}{\mathbb{E}[e^{r\bar{\xi}_{\zeta-}}]}
        \mathbb{E}
        \left [
          \left (
            e^{r(y+\bar{\xi}_{\zeta-})}-K\mathbb{E}\left [e^{r\bar{\xi}_{\zeta-}}\right ]
          \right )^{+}
        \right ],
      \end{align}
      where $\zeta$ is the killing time of the process $\xi$,
      which has exponential distribution with rate $q$,
      and $\bar{\xi}_{t}=\sup_{0\leq s\leq t}\xi_{s}$.
      We remark that $\mathbb{E}\left [e^{r\bar{\xi}_{\zeta-}}\right ]=\frac{\kappa (q,0)}{\kappa(q,-r)}$. The optimal stopping time is given by
      \[
        \sigma^{*}=\inf \{t\geq 0: \xi_{t}\geq c^{*}\}
      \]
      where
      \[
        c^{*}=\frac{1}{r} \log K \frac{\kappa(q,0)}{\kappa(q,-r)}.
      \]
    \item\label{i:xi:b} 
      When $-(\delta+1-\alpha)<r<0$,  the solution is given by 
      \begin{align}
        w(y)
        =\frac{1}{\mathbb{E}[e^{r\underline{\xi}_{\zeta-}}]}
      \mathbb{E}\left [
        \left (
          e^{r(y+\underline{\xi}_{\zeta-})}-K\mathbb{E}\left [e^{r\underline{\xi}_{\zeta-}}\right ]
        \right )^{+}
      \right ]
      \end{align}
      where $\underline{\xi}_{t}=\inf_{0\leq s\leq t}\xi_{s}$, and
      $\mathbb{E}\left [e^{r\underline{\xi}_{\zeta-}}\right ]
      =\frac{\hat{\kappa} (q,0)}{\hat{\kappa}(q,r)}$.
      The optimal stopping time is given by
      \[
        \sigma^{*}=\inf \{t\geq 0: \xi_{t}\leq -c^{*}\}
      \]
      where
      \[
        c^{*}
        =
        \frac{1}{\lvert r\rvert}\log K \frac{\hat{\kappa}(q,0)}{\hat{\kappa}(q,r)}.
      \]
    \item\label{i:xi:c} 
      When $r>\delta$ or $r < -(\delta+1-\alpha)$,
      $w(y)=\infty$ for all $y$.
  \end{enumerate}
\end{theorem}

\begin{proof}
  Theorem~1 in \cite{ME} provides a solution for the valuation of the perpetual
  American call option for a general Lévy process, expressed in terms of the
  moment generating function of its supremum.
  The proof follows from an application of this result, the main prerequisite
  for which is to check the inequality
  $\EE e^{r\xi_1} < 1$.

  \begin{enumerate}
    \item
      Take $r>0$.
      Condition $\EE e^{r\xi_1} < 1$ is equivalent to $\Psi(-ir) < 0$,
      where $\Psi$ is the characteristic exponent of the killed
      Lévy process $\xi$.

      The first zero of $\Psi(-ir)$ occurs at $\delta$, and so
      in this case
      $\EE e^{r \xi_1} < 1$ if and only if $0<r<\delta$.

      The value function $w$ is now expressed in terms of the moment
      generating function of the overall supremum of $r\xi$, which is
      given by \cite[Theorem~6.15]{Kyprianou} in terms of the Wiener-Hopf factor:
      \[
        \mathbb{E}\left [e^{r\bar{\xi}_{\zeta-}}\right ]
        =\frac{\kappa(q,0)}{\kappa(q,-r)}.
      \] 
      The result follows from \cite{ME}.

    \item
      Take now $r<0$. In this case, $\EE e^{r \xi_1} < 1$ if and only
      if $0< -r < \delta+1-\alpha$. Since $r<0$, the value function
      is expressed in terms of
      \[
        \mathbb{E}[e^{r\underline{\xi}_{\zeta-}}]
        =
        \frac{\hat{\kappa}(q,0)}{\hat{\kappa}(q,r)},
      \]
      and the result again follows from \cite{ME}.

    \item
      In this case, it holds that $\EE[e^{r\xi_1}] > 1$, and so the
      result follows from \cite[Theorem~1(c) and~(d)]{ME}, where
      Mordecki shows that arbitrarily large values can be obtained
      by stopping at deterministic times. \qedhere
  \end{enumerate}
\end{proof}

\begin{corollary}\label{c:w}
  When $0<r<\delta$, we can express
  \[
    w(y)
    =
    \kappa(q,-r)
    \int_0^\infty
    \left(e^{r(y+z)} - K \frac{\kappa(q,0)}{\kappa(q,-r)}\right)^+
    u(z)\, \dd z,
    \quad y \in \RR,
  \]
  where
  \[
    u(x)
    =
    \frac{1}{\Gamma(\alpha\rho)}
    e^{-\delta x} (1-e^{-x})^{\alpha\rho-1}
    \HG(\delta-\alpha\rhoh, \delta, \alpha\rho, 1-e^{-x})
    , \quad x \ge 0,
  \]
  and $\HG$ is the hypergeometric function.
  When $-(\delta+1-\alpha)<r<0$, we have
  \[
    w(y)
    =
    \hat{\kappa}(q,r)
    \int_0^\infty
    \left( e^{r(y-z)} - K \frac{\hat{\kappa}(q,0)}{\hat{\kappa}(q,r)} \right)^+
    \hat{u}(z)\, \dd z,
    \quad y\in \RR,
  \]
  where
  \[
    \hat{u}(x)
    =
    \frac{1}{\Gamma(\alpha\hat{\rho})}
    e^{-(1-\delta)x}
    (1-e^{-x})^{\alpha\hat{\rho}-1}
    \HG(\alpha-\delta, \alpha\rhoh-\delta, \alpha\rhoh, 1-e^{-x}),
    \quad x \ge 0.
  \]
\end{corollary}
\begin{proof}
  The random variable $\bar{\xi}_{\zeta-}$ has Laplace transform given by
  \[
    \lambda \mapsto \frac{\kappa(q,0)}{\kappa(q,\lambda)}, 
    \quad \tRe \lambda > \max(-\delta, \delta-\alpha),
  \]
  and the double beta subordinator with Laplace exponent $\kappa(q,\cdot)$ has
  a renewal density $u$ whose Laplace transform is given by
  $\lambda \mapsto \frac{1}{\kappa(q,\lambda)}$.
  The function $u$ is the convolution of the two functions
  \begin{align*}
    u_1(x) 
    &= \frac{1}{\Gamma(\delta)} 
    e^{-(\alpha-\delta)x}(1-e^{-x})^{\delta-1}, 
    \quad x\ge0, \text{ and}\\
    u_2(x) &= \frac{1}{\Gamma(\alpha\rho-\delta)} 
    e^{-\delta x}(1-e^{-x})^{\alpha\rho-\delta-1},
    \quad x\ge0,
  \end{align*}
  which are the renewal densities of Lamperti-stable subordinators 
  corresponding, respectively, to the Laplace exponents
  $\lambda\mapsto \frac{\Gamma(\alpha\rho+\lambda)}{\Gamma(\delta+\lambda)}$
  and $\lambda \mapsto \frac{\Gamma(\alpha+\lambda)}{\Gamma(\alpha-\delta+\lambda)}$.
  Computing the convolution of $u_1$ and $u_2$ yields the expression
  in the statement.

  The proof in the case $r<0$ is very similar, but we instead need to compute
  the convolution of functions
  \begin{align*}
    \hat{u}_1(x)
    &=
    \frac{1}{\Gamma(\delta)} (e^x-1)^{\delta-1},
    \quad x \ge 0, \text{ and } \\
    \hat{u}_2(x)
    &=
    \frac{1}{\Gamma(\alpha\rhoh-\delta)}
    e^{-(1-\alpha+\delta)x}(1-e^{-x})^{\alpha\rhoh-\delta-1},
    \quad x \ge 0,
  \end{align*}
  in order to compute the density of $\underline{\xi}_{\zeta-}$.
\end{proof}
Expressions for $u$ and $\hat{u}$ in general (as potential densities of
a double beta subordinator) appear in \cite[Theorem~2.1]{DHG}, but
for our particular parameter set, the ones above are simpler.

\section{Solution of the optimal stopping problem}
\label{s:os-sol}

In this section, we return to our original problem \eqref{ogp4X}, which we can
rewrite as
\begin{equation}\label{e:os-Xdag}
  v(x) = \sup_{\tau \in \mathscr{S}_{\mathbb{F}^\dag}} [ g(X^\dag_t) ],
\end{equation}
where
$\mathscr{S}_{\mathbb{H}}$ represents the set of stopping times associated
with some filtration $\mathbb{H}$, and, as defined earlier,
$\mathbb{F}^\dag$ is the natural enlargement of the filtration of $X^\dag$.
We also recall the gain function $g$ defined
in the introduction:
\begin{equation*}
  g(x)
  =
  \begin{cases}
    (x^{r}-K)^{+}, & x > 0\\
    0, & x \le 0.
  \end{cases}
\end{equation*}
The reason for our introduction of the notation $\mathscr{S}_{\mathbb{H}}$
is that we will have cause to work with several filtrations as we connect
\eqref{e:os-Xdag} to the problem we considered in \autoref{s:os-lamperti}.

The following lemma will be important in understanding the structure
of the solution of the optimal stopping problem.
\begin{lemma}\label{l:equiv}
  When $-(\delta+1-\alpha)< r<\delta$, 
  $\mathbb{E}_{x}\left [\sup_{t\ge 0}g\left (X_{t}^\dag\right )\right ]<\infty$
  for all $x \in \RR$.
\end{lemma}
\begin{proof}
  Assume initially that $x>0$.
  First, observe that $\sup_{t\ge 0} g(X_t^\dag)$
  can only be attained at a time when $X$ is positive. The time-change
  by $C^{-1}$, which yields $Y$, does not remove any such times.
  Applying this and an elementary bound for $g$, we obtain:
  \begin{align}
    \mathbb{E}_{x}\left [\sup_{t\ge 0}g({X}^\dag_{t})\right ]
    =\mathbb{E}_{x}\left [\sup_{t \ge 0}g(Y_{t})\right ]
    \le \mathbb{E}_{x}\left [\sup_{t \ge 0}Y_{t}^r \Ind_{\{Y_t\ne 0\}} \right ].
    \label{X2Y}
  \end{align}
  Given a path of $Y$ starting at $x$, the corresponding path of $\xi$ starts at 
  $\log x$ and is a deterministic space and time transformation,
  which implies that
  \begin{align}
    \mathbb{E}_{x}\left [\sup_{t \ge 0}Y_{t}^r \Ind_{\{Y_t\ne 0\}} \right ]
    = \mathbb{E}_{\log x}\left [\sup_{s \ge 0}e^{r\xi_{s}}\right ].
    \label{Y2xi}
  \end{align}
  We next need to show that the right hand side of \eqref{Y2xi} is finite.
  For this, we recall that by Theorem~1(a) in \cite{ME},
  $\mathbb{E}_{0}\left [\sup_{s \ge 0} e^{r\xi_{s}}\right ]<\infty$ 
  is true if $\mathbb{E}_{0}\left [e^{r\xi_{1}}\right ]<1$ holds, 
  and as outlined in the proof of \autoref{t:main-lamperti},
  this can be reduced to the condition that 
  $-(\delta+1-\alpha)< r<\delta$.
  This completes the proof in the case $x>0$.

  We now consider the possibility that $x<0$. In this case,
  the supremum in question can only
  be positive if it occurs
  after the time $\tau_0^+$ at which $X$ returns above zero.
  This together with the scaling property and a bound on $g$ implies that
  \[
    \EE_x\Bigl[ \sup_{t\ge 0} g(X^\dag_t)\Bigr]
    \le \EE_x \Bigl[ 
      \EE_{X_{\tau_0^+}}\Bigl[ \sup_{t\ge 0} (X^\dag_t)^r \Ind_{\{X_t>0\}} \Bigr]
    \Bigr]
    =
    \EE_x\bigl[ X_{\tau_0^+}^r \bigr] \EE_1\Bigl[\sup_{t\ge 0} (X^\dag_t)^r \Ind_{\{X_t>0\}}\Bigr]
  \]
  The second expectation is equal to the left-hand side of \eqref{Y2xi} at $x=1$,
  and hence is finite.
  For the first expectation, a result of \citet[Lemma~3]{Rog72} is useful; 
  note that in applying it,
  we correct for the error in that lemma ($q$, which is denoted $\rho$ here, should be
  $1-q$). This gives:
  \[
    \EE_x\bigl[ X_{\tau_0^+}^r \bigr]
    = \frac{\sin\pi\alpha\rho}{\pi} (-x)^{\alpha\rho}
    \int_0^\infty y^{r-\alpha\rho}(y-x)^{-1} \, \dd y,
  \]
  and thanks to the conditions on $r$ given in the statement,
  this is finite for all $x$.

  If $x=0$, then $X^\dag_t = 0$ and the result follows.
\end{proof}

This lemma leads to a key theorem about the stopping problem.
\begin{theorem}\label{t:os-theory}
  When $0<r<\delta$
  or
  $-(\delta+1-\alpha)<r<0$,
  there exists a set $D\subset \left (0, \infty \right )$ such that
  \[
    \tau_{D}\coloneqq \inf \{t\geq 0 :  X^\dag_{t}\in D\}
  \]
  is an optimal stopping time in \eqref{e:os-Xdag}.
\end{theorem}
\begin{proof}
  We begin by considering the case $r > 0$,
  and we will apply \cite[Corollary~2.9]{GP}. That result is formulated
  for finite-horizon stopping problems, but since
  $\lim_{t\to\infty} X^\dag_t = 0$, Remark~2.10 in the same work notes that
  the application remains valid in this case.
  We remark that what we call a `stopping time' is known as a `Markov time'
  in \cite{GP}.

  Let
  \[
    D = \{ x \in \RR : v(x) = g(x)\} .
  \]
  The result of \cite{GP} states that the first entrance time $\tau_D$ of $X^\dag$
  into $D$ is optimal for \eqref{e:os-Xdag} if we can verify the following points:
  \begin{enumerate}
    \item
      $\EE_x[\sup_{t\ge 0} g(X_t^\dag)] < \infty$.
      This is \autoref{l:equiv}.
    \item
      $g$ is upper semicontinuous.
      This holds because $g$ is continuous when $r>0$.
    \item
      $v$ is lower semicontinuous.

      As mentioned in Section~\ref{s:kpcsp}, we can regard $X$
      under $\PP_x$ as being defined by
      $X_t = \lvert x \rvert \tilde{X}_{t\lvert x\rvert ^{-\alpha}}$,
      where $\tilde{X}$ has the law of $X$ started from $\sgn x$. This allows
      us to reduce the treatment of all the measures 
      $(\PP_x)_{x\in \RR\setminus \{0\}}$ to just that of $\PP_1$ and $\PP_{-1}$,
      and the identity extends also to $X^\dag$, since it too is self-similar.
      In particular, it means that we can write
      \[
        \EE_x [ g(X^\dag_\tau)] 
        = \EE_{\sgn x}[ g(\lvert x\rvert X^\dag_{\tau\lvert x\rvert^{-\alpha}})],
      \]
      for any $x\ne 0$ and random time $\tau$. This will be quite useful.

      Fix $x\ne 0$.
      Take $\epsilon > 0$, and let $\tau_\epsilon$ be a stopping time satisfying
      \[
        \EE_x[ g(X^\dag_{\tau_\epsilon})] \ge v(x) - \epsilon/2.
      \]
      Now, the right-continuity of $X^\dag$, the
      dominated convergence theorem and \autoref{l:equiv}
      together imply that
      \[
        \lim_{\eta \downarrow 0} \EE_x[ g(X^\dag_{\tau_\epsilon + \eta})] 
        = \EE_x[ g(X_{\tau_\epsilon}^\dag)],
      \]
      which means that, if we set $\tau_\epsilon' = \tau_\epsilon+\eta$
      for $\eta>0$ small enough,
      \[
        \EE_x[ g(X^\dag_{\tau_\epsilon'}) ] \ge v(x) - \epsilon.
      \]
      Since
      $\tau_\epsilon' = \lim_{\eta'\uparrow \eta} (\tau_\epsilon + \eta')$
      gives an expression in terms of an increasing limit of stopping times,
      and $X^\dag$ is quasi-left-continuous by \autoref{l:standard},
      it follows that $X^\dag$ is continuous at $\tau_\epsilon'$ almost surely.

      Take $(x_n)_{n\ge 0}$ to be a sequence converging to $x$ and, from now
      on, take $n$ sufficiently large that
      $\sgn x_n = \sgn x$. We first observe that
      \[
        v(x_n) \ge \EE_{x_n}[ g(X^\dag_{\tau'_\epsilon})]
        = \EE_{\sgn x}[ g(\lvert x_n\rvert X^\dag_{\lvert x_n\rvert^{-\alpha} \tau'_\epsilon})].
      \]
      Applying Fatou's lemma and using the continuity of $X^\dag$ at $\tau'_\epsilon$,
      we obtain:
      \begin{align*}
        \liminf_{n\to\infty} v(x_n)
        &\ge
        \EE_{\sgn x}[ \liminf_{n\to \infty} g(\lvert x_n\rvert X^\dag_{\lvert x_n\rvert^{-\alpha} \tau'_\epsilon})]
        \\
        &\ge
        \EE_{\sgn x}[
        \min\{
          g(\lvert x\rvert X^\dag_{\lvert x\rvert^{-\alpha} \tau'_\epsilon}),
          g(\lvert x\rvert X^\dag_{\lvert x\rvert^{-\alpha} \tau'_\epsilon- })
        \}
        ]
        \\
        &= \EE_x[ \min\{ g(X^\dag_{\tau'_\epsilon}), g(X^\dag_{\tau'_\epsilon -}) \}] \\
        &= \EE_x[ g(X^\dag_{\tau'_\epsilon})] \ge v(x) - \epsilon.
      \end{align*}
      Finally, letting $\epsilon\to 0$, we have
      \[
        \liminf_{n\to \infty} v(x_n) \ge v(x),
      \]
      which is the lower semicontinuity we require for $x\ne 0$.

      Take $x = 0$, a stopping time $\tau$
      and a sequence $(x_n)_{n\ge 0}$
      convenging to $0$.
      Then,
      \[
        \liminf_{n\to\infty} v(x_n)
        = \liminf_{n\to\infty} \EE_{x_n}[ g(X^\dag_\tau) ]
        \ge 0 = v(0).
      \]
      Therefore,
      $v$ is lower semicontinuous also at $0$, which completes the proof
      of this part.
  \end{enumerate}

  Now turn to the case $r<0$, in which $g$ is no longer continuous.
  Our approach will be to make a space-time transformation of $X$,
  known as the Riesz–Bogdan–Żak transformation \cite[Theorem~3.1]{Kyp-deep1}.
  Let
  $\theta(t) = \Bigl( \int_0^\cdot \lvert X_u\rvert^{-2\alpha} \, \dd u\Bigr)^{-1}(t)$,
  and define $X_t^\circ = 1/X_{\theta(t)}$.
  The associated probability measures are $\PP^\circ_x = \PP_{1/x}$
  for $x\ne 0$.

  The process $X^\circ$ is a Markov process which is self-similar of index $\alpha$,
  in the sense that it satisfies \eqref{e:scaling} for $y\in\RR\setminus\{0\}$.
  Note that $\theta^{-1}(T_0) = \infty$ a.s., so the paths of $X$ and
  $X^\circ$ are in correspondence up to $T_0$, the first time $X$ hits zero, if
  at all.

  It is not hard to show that 
  $A^\circ_t = \int_0^t \omega(X^\circ_s)\, \dd s = A_{\theta(t)}$,
  and we can define $T^\circ = \inf\{ t\ge 0: A^\circ_t > \mathrm{e}_1\}$,
  where $\mathrm{e}_1$ is the independent random variable of rate 1
  which was used in the definition of $T$.
  Note that $T = \theta(T^\circ)$.

  Let $X^{\circ\dag}_t = \frac{1}{X_{\theta(t)}} \Ind_{\{\theta(t) < T\}} = X^\circ_t\Ind_{\{t<T^\circ\}}$.
  For this process we can also define the probability measure $\PP^\circ_0$ under
  which $X^{\circ\dag}$ remains at $0$ for all time.
  In this way, we obtain a real self-similar Markov process of index $\alpha$,
  following essentially the same argument as in the proofs
  of \autoref{l:standard} and \autoref{p:ssmp}.

  Every stopping time $\tau$ for $\mathbb{F}^\dag$ 
  can be replaced by a stopping time
  $\tau' = \tau\wedge T_0$ without reducing the payoff, and
  $\theta^{-1}(\tau')$ is a stopping time for the filtration
  $\mathbb{F}^{\circ\dag} = \bigl( \mathscr{F}^\dag_{\theta(t)} \bigr)_{t\ge 0}$,
  which is also the natural enlargement of the filtration of $X^{\circ\dag}$.
  Defining
  \[
    g^\circ(x)
    =
    \begin{cases}
      (x^{-r} - K)_+, & x > 0, \\
      0, & x \le 0,
    \end{cases}
  \]
  this discussion gives us that
  \begin{equation}
    \label{e:vo-v}
    v^\circ(x) \coloneqq \sup_{\tau \in \mathscr{S}_{\mathbb{F}^{\circ\dag}}}
    \EE^\circ_x[ g^\circ(X_\tau) ]
    = v(1/x),
  \end{equation}
  for $x\ne 0$, and $v^\circ(0) = 0$.
  From this, we can use the result of \cite{GP} to conclude that the the stopping
  region for this new, equivalent optimal stopping problem is
  \[
    D^\circ = \{ x \in \RR : v^\circ(x) = g^\circ(x) \},
  \]
  provided we can verify the analogues of points (i--iii) above.

  Point (i) follows immediately, and point (ii) holds because $g^\circ$
  is continuous when $r<0$. Point (iii) must be verified, but the proof above
  works with essentially no change, since it uses only the fact that
  the process in question is self-similar, strong Markov and has no fixed
  jumps, and this is true for $X^{\circ\dag}$ just as for $X^\dag$.

  Defining $D = \{ 1/x : x \in D^\circ\setminus\{0\} \} \cup\{0\}$,
  it follows from \eqref{e:vo-v} and the definition of $X^{\circ\dag}$
  that the entrance time of $D$ is optimal for \eqref{e:os-Xdag}.

  Finally, it is simple to see that $0$ belongs to the sets $D$ defined
  above, since $v(0) = 0 = g(0)$ in all cases. However, this point, 
  and more importantly
  any points in $(-\infty,0]$,
  can
  be removed from $D$ without changing the value function, as we now show.
  Set $D' = D \cap (0,\infty)$, and then consider the following calculation:
  \begin{align*}
    \EE_x [ g(X_{\tau_D}) \Ind_{\{\tau_D<T\}} ]
    &= \EE_x [ g(X_{\tau_D}) \Ind_{\{\tau_D<T\}} \Ind_{\{ X_{\tau_D}>0\}}]
    + \EE_x [ g(X_{\tau_D}) \Ind_{\{\tau_D<T\}} \Ind_{\{X_{\tau_D}\le 0\}}] \\
    &= \EE_x [ g(X_{\tau_{D'}}) \Ind_{\{\tau_{D'}<T\}} \Ind_{\{ \tau_D = \tau_{D'} \}} ] + 0 \\
    &\le \EE_x [ g(X_{\tau_{D'}}) \Ind_{\{\tau_{D'}<T\}} ].
  \end{align*}
  From this it follows that, if $D \not\subset (0,\infty)$, we can replace it with $D'$ and obtain
  at least as good a value; indeed, since $\tau_D$ is optimal, the values
  obtained from $\tau_D$ and $\tau_{D'}$ will actually be equal.
\end{proof}

Our main result, which implies \autoref{t:main} appearing in the introduction,
is as follows. Recall that $T$ is the killing time of the process $X^\dag$.

\begin{theorem}\label{t:os-X-full}
  Let $x\in \RR\setminus\{0\}$.
  The solution of the optimal stopping problem \eqref{e:os-Xdag} is given as follows.
  \begin{enumerate}
    \item If $0 < r<\delta$, 
      then the optimal stopping time is given by
      \[\tau^{*}=\inf \{t\geq 0: X_{t}\geq b^{*}\}\]
      where
      \[
        b^{*}
        = \left(K \frac{\kappa(q,0)}{\kappa(q,-r)}\right)^{1/r}.
      \]
      Moreover $\mathbb{E}_1\!\left [\bar{X}_{T}^{r}\right ]<\infty$ 
      and the optimal value is given for $x>0$ by
      \begin{align*}
        v(x)
        =
        \frac{1}{\EE_1[\bar{X}_T^r]}
        \mathbb{E}_1\!\left [
          \left (
            \bigl(x\bar{X}_{T}\bigr)^{r}-K\mathbb{E}_1\!\left [\bar{X}_{T}^{r}\right ]
          \right )^{+}
        \right ],
      \end{align*}
      where $\bar{X}_{t}=\sup_{s\le t}X_{s}$. 
    \item 
      If $-(\delta+1-\alpha)<r<0$,
      then the optimal stopping time is given by
      \[\tau^{*}=\inf \{t\geq 0: 0<X_{t}\leq 1/{b^{*}}\}\]
      where
      \[
        b^{*}
        = \left(K \frac{\hat{\kappa}(q,0)}{\hat{\kappa}(q,r)}\right)^{1/\lvert r\rvert}.
      \]
      Moreover, $\mathbb{E}_1\!\left [J(X)_{T}^{r}\right ]<\infty$ 
      and the optimal value is given for $x>0$ by
      \begin{align*}
        v(x)
        =
        \frac{1}{\EE_1[J(X)_T^{r}]}
        \mathbb{E}_1\!\left [
          \left (
            \left (x J(X)_{T}\right )^{r}-K\mathbb{E}_1\!\left [J(X)_{T}^{r}\right ]
          \right )^{+}
        \right ],
      \end{align*}
      where $J(X)_t = \inf\{ X_s: s\le t, X_s > 0 \}$.
    \item
      \label{i:full-3}
      If $r > \delta$ or $r < -(\delta+1-\alpha)$, then
      $v(x) = \infty$.
  \end{enumerate}
\end{theorem}

\begin{proof}
  Since $g(x)=0$ for $x\le 0$, it is never optimal to stop when $X$ is negative.
  The processes $X$ and $Y$ have the same range when
  restricted to $(0,\infty)$, 
  so the optimal stopping problem \eqref{e:os-Xdag} is equivalent
  to
  \begin{align*}
    v(x)
    =
    \sup_{\tau' \in \mathscr{S}_{\mathbb{F}^\dag\circ C^{-1}}}
    \mathbb{E}_{x}\left [g\left (Y_{\tau'}\right )\right ],
  \end{align*}
  where we recall 
  $\mathbb{F}^\dag\circ C^{-1} = \bigl( \mathscr{F}^\dag_{C^{-1}_t} \bigr)_{t\ge 0}$.
  Moreover, as already outlined, the pssMp $Y$ corresponds to $\xi$ under a time
  and space change, which implies that
  \begin{align}
    v(y)
    =
    \sup_{\tau''\in \mathscr{S}_{\mathbb{F}^\dag\circ C^{-1}\circ T}}
    \mathbb{E}_{y}\left [g\left (e^{\xi_{\tau''}}\right )\right ],\label{p4xi}
  \end{align}
  where $y=\log x$ and here again we have
  $\mathbb{F}^\dag\circ C^{-1}\circ T 
  = \bigl( \mathscr{F}^\dag_{C^{-1}_{T(t)}} \bigr)_{t\ge 0}$.

  We know from \autoref{t:os-theory} that the hitting time $\tau_{D}$ 
  of some set $D \subset (0,\infty)$ is optimal for
  \eqref{e:os-Xdag},
  and hence that the first passage time $\sigma_H$ of $\xi$
  into set $H = \{\log x: x \in D\}$ is optimal for
  \eqref{p4xi}.

  The solution in \autoref{t:main-lamperti} showed that the solution
  of the problem \eqref{callop} for $\xi$ is given by a hitting time
  $\sigma^*$.
    In \eqref{p4xi}, we optimise over $\mathscr{S}_{\mathbb{F}^\dag\circ C^{-1} \circ T}$,
  and in \eqref{callop} we optimise over $\mathscr{S}_{\mathbb{G}}$.
  The former set of stopping times is larger, since the filtration
  contains information about the times that $X^\dag$ spends below zero.
  Since $\sigma_{H} \in \mathscr{S}_{\mathbb{G}}$, it follows that it
  is also optimal for \eqref{callop}.
  Comparing the form of $\sigma^*$ with $\sigma_H$ and hence with
  the original stopping region $D$, we obtain that
  \[
    D = \left [ \left(K \frac{\kappa(q,0)}{\kappa(q,-r)}\right)^{1/r}, \infty \right )
    \text{ if } 0 < r<\delta,
  \]
  and
  \[
    D = \left (0, \left(
        \frac{\hat{\kappa}(q,-r)}{K\hat{\kappa}(q,0)}
    \right)^{1/\lvert r\rvert} \right ]
    \text{ if } -(\delta+1-\alpha)< r<0.
  \]
  For part \ref{i:full-3}, we turn to the corresponding part of
  \autoref{t:main-lamperti}, where it is shown that arbitrarily large
  values of $w$ can be obtained by stopping at a deterministic time,
  say $t_0$. Since time $t_0$ for $\xi$ corresponds to time
  $C^{-1}_{T(t_0)}$ for $X$, this is a viable time at which
  to stop $X$. It follows that $v(x) = \infty$ for $x>0$.
  When $x<0$, one can first wait for $X$ to pass above zero,
  which happens without being killed with positive probability, and then
  act as above, again attaining unbounded values.
\end{proof}

\begin{remark}
  \label{r:int}
  When $x>0$, the value function can be expressed as an integral by noting that
  $v(x) = w(\log x)$, where $w$ is the value function computed in
  \autoref{c:w}.

  When $x< 0$, it is still possible to obtain a (double) integral expression
  for $v$, though it is less direct.
  In this case, $X^\dag$ is started
  below zero and we should wait for it to either be killed (with
  probability $p$ independent of $x$) or jump back above zero (with probability
  $1-p$). Let $\tau_0^+ = \inf\{ t \ge 0: X^\dag_t \ge 0 \}$. Then,
  for all $x<0$,
  \begin{align*}
    v(x) &= \EE_x[v(X^\dag_{\tau_0^+})] \\
    &= (1-p) \int_0^\infty v(y) \frac{\sin \pi\alpha\rho}{\pi}
    (-x)^{\alpha\rho} y^{-\alpha\rho} (y-x)^{-1} \, \dd y,
  \end{align*}
  where the integral expression comes from the same
  result of \cite{Rog72} which we used in \autoref{l:equiv}.
  In conjunction with the expression for $w$
  given in \autoref{c:w},
  this can be used to write $v$ as a double integral suitable
  for numerical computation.
\end{remark}

\section{Remark on a variant stopping problem}
\label{s:variants}

In this section we briefly describe how a gain function akin to that of
a put option
requires a different analysis, despite the superficial similarities.
We consider the optimal stopping problem
\begin{align}
  v(x)=\sup_{\tau}\mathbb{E}_{x}\left [g(X_{\tau})\mathbbm{1}_{\{\tau<T\}}\right ],
  \qquad
  g(x) = (K-x)^+,\label{vput}
\end{align}
where $K\in \RR$.
Since $g(x)$ may be positive for $x<0$, it no longer makes sense to erase
the sojourns of $X$ in $(-\infty,0)$. Instead, we may describe the problem
using the so-called Lamperti-Kiu transform. This gives $X^\dag$
in terms of
a Markov additive process (MAP), which is a process $(\xi,J)$ on $\RR\times\{\pm 1\}$
obtained by
\[
  \xi_{t} 
  = \log \lvert X^\dag_{T(t)}\rvert 
  \text{ and } 
  J_{t}=\sgn X^\dag_{T(t)}, \quad t\geq 0,
\]
where $T$ is a time-change.

The process $(\xi,J)$ corresponding to $X^\dag$ will be killed at a rate $\omega(\xi_t,J_t)$,
where
\begin{equation*}
  \omega(y,j)
  =
  \begin{cases}
    0, &j=1, \\
    k, &j=-1,
  \end{cases}
\end{equation*}
and the problem \eqref{vput} will correspond to
\begin{equation*}
  v(x)
  =
  v(y,j)
  =
  \sup_{\sigma}
  \mathbb{E}_{y, j}\left [ g(\xi_{\sigma}, J_{\sigma}) \Ind_{\{t<\zeta\}} \right ],
\end{equation*}
where $(y,j)=(\log |x|,\sgn x)$, $\zeta$ is the killing time of $(\xi,J)$ and
\begin{equation*}
  g(y, j)
  =
  \begin{cases}
    \left (K-e^{y}\right )^{+} , &j=1, \\
    \left (K+e^{y}\right )^{+} , &j=-1.
  \end{cases}
\end{equation*}
The MAP $(\xi,J)$ can be described
explicitly in terms of its matrix exponent.
Unfortunately, though this translation to a MAP problem is relatively simple
to describe, two new issues arise. 
The first is that the presence of $J$-dependent killing
means that the matrix Wiener-Hopf factorisation of $(\xi,J)$, which is known
when $k=0$ \cite{Kyp-deep1,Kyp-deep2}, is no longer evident.
The second is that, even if the factorisation were known, the theory of optimal
stopping is much less developed for these processes, outside of the spectrally
negative case \cite{CKV}.

\paragraph{Acknowledgments}
A version of this work appears as part of the
doctoral thesis of the third author.
All authors would like to thank two anonymous referees, whose comments have
greatly improved the argumentation and presentation of this work.

\bibliographystyle{abbrvnat}
\bibliography{refs}

\begin{thebibliography}{27}
\providecommand{\natexlab}[1]{#1}
\providecommand{\url}[1]{\texttt{#1}}
\expandafter\ifx\csname urlstyle\endcsname\relax
  \providecommand{\doi}[1]{doi: #1}\else
  \providecommand{\doi}{doi: \begingroup \urlstyle{rm}\Url}\fi

\bibitem[Albrecher and Lautscham(2013)]{Al2}
H.~Albrecher and V.~Lautscham.
\newblock From ruin to bankruptcy for compound {P}oisson surplus processes.
\newblock \emph{Astin Bull.}, 43\penalty0 (2):\penalty0 213--243, 2013.
\newblock ISSN 0515-0361.
\newblock \doi{10.1017/asb.2013.4}.

\bibitem[Albrecher et~al.(2011)Albrecher, Gerber, and Shiu]{Al}
H.~Albrecher, H.~U. Gerber, and E.~S.~W. Shiu.
\newblock The optimal dividend barrier in the gamma-omega model.
\newblock \emph{Eur. Actuar. J.}, 1\penalty0 (1):\penalty0 43--55, 2011.
\newblock ISSN 2190-9733.
\newblock \doi{10.1007/s13385-011-0006-4}.

\bibitem[Bichteler(2002)]{Bic}
K.~Bichteler.
\newblock \emph{Stochastic Integration with Jumps}, volume~89 of
  \emph{Encyclopedia of Mathematics and its Applications}.
\newblock Cambridge University Press, Cambridge, 2002.
\newblock ISBN 0-521-81129-5.
\newblock \doi{10.1017/CBO9780511549878}.

\bibitem[Blumenthal and Getoor(1968)]{BG-mp}
R.~M. Blumenthal and R.~K. Getoor.
\newblock \emph{Markov Processes and Potential Theory}, volume Vol. 29 of
  \emph{Pure and Applied Mathematics}.
\newblock Academic Press, New York-London, 1968.

\bibitem[Bogdan et~al.(2003)Bogdan, Burdzy, and Chen]{BBC-cens}
K.~Bogdan, K.~Burdzy, and Z.-Q. Chen.
\newblock Censored stable processes.
\newblock \emph{Probab. Theory Related Fields}, 127\penalty0 (1):\penalty0
  89--152, 2003.
\newblock ISSN 0178-8051.
\newblock \doi{10.1007/s00440-003-0275-1}.

\bibitem[Budd(2018)]{Budd}
T.~Budd.
\newblock The peeling process on random planar maps coupled to an {$O(n)$} loop
  model (with an appendix by {L}inxiao {C}hen).
\newblock \emph{Preprint (arXiv)}, 2018.
\newblock \doi{10.48550/arXiv.1809.02012}.

\bibitem[Caballero and Chaumont(2006)]{CC}
M.~E. Caballero and L.~Chaumont.
\newblock Conditioned stable {L}\'{e}vy processes and the {L}amperti
  representation.
\newblock \emph{J. Appl. Probab.}, 43\penalty0 (4):\penalty0 967--983, 2006.
\newblock ISSN 0021-9002.
\newblock \doi{10.1239/jap/1165505201}.

\bibitem[Caballero et~al.(2011)Caballero, Pardo, and P\'{e}rez]{CPP-explicit}
M.~E. Caballero, J.~C. Pardo, and J.~L. P\'{e}rez.
\newblock Explicit identities for {L}\'{e}vy processes associated to symmetric
  stable processes.
\newblock \emph{Bernoulli}, 17\penalty0 (1):\penalty0 34--59, 2011.
\newblock ISSN 1350-7265.
\newblock \doi{10.3150/10-BEJ275}.

\bibitem[\c{C}a\u{g}lar et~al.(2022)\c{C}a\u{g}lar, Kyprianou, and
  Vardar-Acar]{CKV}
M.~\c{C}a\u{g}lar, A.~Kyprianou, and C.~Vardar-Acar.
\newblock An optimal stopping problem for spectrally negative {M}arkov additive
  processes.
\newblock \emph{Stochastic Process. Appl.}, 150:\penalty0 1109--1138, 2022.
\newblock ISSN 0304-4149,1879-209X.
\newblock \doi{10.1016/j.spa.2021.06.010}.

\bibitem[Czarna et~al.(2020)Czarna, Kaszubowski, Li, and Palmowski]{Cza}
I.~Czarna, A.~Kaszubowski, S.~Li, and Z.~Palmowski.
\newblock Fluctuation identities for omega-killed spectrally negative {M}arkov
  additive processes and dividend problem.
\newblock \emph{Adv. in Appl. Probab.}, 52\penalty0 (2):\penalty0 404--432,
  2020.
\newblock ISSN 0001-8678.
\newblock \doi{10.1017/apr.2020.2}.

\bibitem[Gerber et~al.(2012)Gerber, Shiu, and Yang]{G}
H.~U. Gerber, E.~S.~W. Shiu, and H.~Yang.
\newblock The {O}mega model: from bankruptcy to occupation times in the red.
\newblock \emph{Eur. Actuar. J.}, 2\penalty0 (2):\penalty0 259--272, 2012.
\newblock ISSN 2190-9733.
\newblock \doi{10.1007/s13385-012-0052-6}.

\bibitem[Kaszubowski(2019)]{Ka}
A.~Kaszubowski.
\newblock Omega bankruptcy for different l{\'e}vy models.
\newblock \emph{{\'S}l{\k{a}}ski Przegl{\k{a}}d Statystyczny}, \penalty0 (17
  (23)):\penalty0 31--58, 2019.

\bibitem[K\"{u}hn and van Schaik(2008)]{Kuhn}
C.~K\"{u}hn and K.~van Schaik.
\newblock Perpetual convertible bonds with credit risk.
\newblock \emph{Stochastics}, 80\penalty0 (6):\penalty0 585--610, 2008.
\newblock ISSN 1744-2508.
\newblock \doi{10.1080/17442500802263888}.

\bibitem[Kuznetsov and Pardo(2013)]{KP-hg}
A.~Kuznetsov and J.~C. Pardo.
\newblock Fluctuations of stable processes and exponential functionals of
  hypergeometric {L}\'{e}vy processes.
\newblock \emph{Acta Appl. Math.}, 123:\penalty0 113--139, 2013.
\newblock ISSN 0167-8019.
\newblock \doi{10.1007/s10440-012-9718-y}.

\bibitem[Kyprianou(2014)]{Kyprianou}
A.~E. Kyprianou.
\newblock \emph{Fluctuations of {L}\'{e}vy Processes with Applications}.
\newblock Universitext. Springer, Heidelberg, second edition, 2014.
\newblock ISBN 978-3-642-37631-3; 978-3-642-37632-0.
\newblock \doi{10.1007/978-3-642-37632-0}.
\newblock Introductory lectures.

\bibitem[Kyprianou(2016)]{Kyp-deep1}
A.~E. Kyprianou.
\newblock Deep factorisation of the stable process.
\newblock \emph{Electron. J. Probab.}, 21:\penalty0 Paper No. 23, 28, 2016.
\newblock \doi{10.1214/16-EJP4506}.

\bibitem[Kyprianou et~al.(2010)Kyprianou, Pardo, and Rivero]{Kyp}
A.~E. Kyprianou, J.~C. Pardo, and V.~Rivero.
\newblock Exact and asymptotic {$n$}-tuple laws at first and last passage.
\newblock \emph{Ann. Appl. Probab.}, 20\penalty0 (2):\penalty0 522--564, 2010.
\newblock ISSN 1050-5164.
\newblock \doi{10.1214/09-AAP626}.

\bibitem[Kyprianou et~al.(2014)Kyprianou, Pardo, and Watson]{KPW}
A.~E. Kyprianou, J.~C. Pardo, and A.~R. Watson.
\newblock Hitting distributions of {$\alpha$}-stable processes via path
  censoring and self-similarity.
\newblock \emph{Ann. Probab.}, 42\penalty0 (1):\penalty0 398--430, 2014.
\newblock ISSN 0091-1798.
\newblock \doi{10.1214/12-AOP790}.

\bibitem[Kyprianou et~al.(2018)Kyprianou, Rivero, and
  \c{S}eng\"{u}l]{Kyp-deep2}
A.~E. Kyprianou, V.~Rivero, and B.~\c{S}eng\"{u}l.
\newblock Deep factorisation of the stable process {II}: {P}otentials and
  applications.
\newblock \emph{Ann. Inst. Henri Poincar\'{e} Probab. Stat.}, 54\penalty0
  (1):\penalty0 343--362, 2018.
\newblock ISSN 0246-0203.
\newblock \doi{10.1214/16-AIHP806}.

\bibitem[Kyprianou et~al.(2021)Kyprianou, Pardo, and Vidmar]{DHG}
A.~E. Kyprianou, J.~C. Pardo, and M.~Vidmar.
\newblock Double hypergeometric {L}\'{e}vy processes and self-similarity.
\newblock \emph{J. Appl. Probab.}, 58\penalty0 (1):\penalty0 254--273, 2021.
\newblock ISSN 0021-9002.
\newblock \doi{10.1017/jpr.2020.86}.

\bibitem[Lamperti(1972)]{Lam}
J.~Lamperti.
\newblock Semi-stable {M}arkov processes. {I}.
\newblock \emph{Z. Wahrscheinlichkeitstheorie und Verw. Gebiete}, 22:\penalty0
  205--225, 1972.
\newblock \doi{10.1007/BF00536091}.

\bibitem[Li and Palmowski(2018)]{Li}
B.~Li and Z.~Palmowski.
\newblock Fluctuations of omega-killed spectrally negative {L}\'{e}vy
  processes.
\newblock \emph{Stochastic Process. Appl.}, 128\penalty0 (10):\penalty0
  3273--3299, 2018.
\newblock ISSN 0304-4149.
\newblock \doi{10.1016/j.spa.2017.10.018}.

\bibitem[Mordecki(2002)]{ME}
E.~Mordecki.
\newblock Optimal stopping and perpetual options for {L}\'{e}vy processes.
\newblock \emph{Finance Stoch.}, 6\penalty0 (4):\penalty0 473--493, 2002.
\newblock ISSN 0949-2984.
\newblock \doi{10.1007/s007800200070}.

\bibitem[Peskir and Shiryaev(2006)]{GP}
G.~Peskir and A.~Shiryaev.
\newblock \emph{Optimal Stopping and Free-Boundary Problems}.
\newblock Lectures in Mathematics ETH Z\"{u}rich. Birkh\"{a}user Verlag, Basel,
  2006.
\newblock ISBN 978-3-7643-2419-3; 3-7643-2419-8.

\bibitem[Rodosthenous and Zhang(2018)]{ZH}
N.~Rodosthenous and H.~Zhang.
\newblock Beating the omega clock: an optimal stopping problem with random
  time-horizon under spectrally negative {L}\'{e}vy models.
\newblock \emph{Ann. Appl. Probab.}, 28\penalty0 (4):\penalty0 2105--2140,
  2018.
\newblock ISSN 1050-5164.
\newblock \doi{10.1214/17-AAP1322}.

\bibitem[Rogozin(1973)]{Rog72}
B.~A. Rogozin.
\newblock The distribution of the first hit for stable and asymptotically
  stable walks on an interval.
\newblock \emph{Theory Probab. Appl.}, 17\penalty0 (2):\penalty0 332--338,
  1973.
\newblock \doi{10.1137/1117035}.

\bibitem[Watson(2013)]{Alex}
A.~R. Watson.
\newblock \emph{Stable processes}.
\newblock PhD thesis, University of Bath, 2013.

\end{thebibliography}

\end{document}